\documentclass[reqno,11pt]{amsart}
\usepackage[margin=1in]{geometry}

\usepackage{amsthm, amsmath, amssymb, bm}

\usepackage[dvipsnames]{xcolor}

\usepackage{tikz}

\usepackage[utf8]{inputenc}
\usepackage[T1]{fontenc}

\usepackage[textsize=small,backgroundcolor=orange!20]{todonotes}

\usepackage[hidelinks]{hyperref}
\usepackage{url}

\usepackage[noabbrev,capitalize]{cleveref}
\crefname{equation}{}{}

\usepackage{xcolor,graphics}

\newtheorem{theorem}{Theorem}[section]

\newtheorem{lemma}[theorem]{Lemma}

\newtheorem{corollary}[theorem]{Corollary}

\newtheorem*{question*}{Question} \Crefname{question}{Question}{Questions}

\theoremstyle{definition}
\newtheorem{definition}[theorem]{Definition}

\theoremstyle{remark}
\newtheorem*{remark}{Remark}

\newcommand{\abs}[1]{\left\lvert#1\right\rvert}
\newcommand{\norm}[1]{\left\lVert#1\right\rVert}

\newcommand{\floor}[1]{\left\lfloor #1 \right\rfloor}
\newcommand{\ceil}[1]{\left\lceil #1 \right\rceil}

\newcommand{\FF}{\mathbb{F}}
\newcommand{\RR}{\mathbb{R}}

\newcommand{\TT}{\mathbb{T}}

\newcommand{\NN}{\mathbb{N}}
\newcommand{\ZZ}{\mathbb{Z}}

\newcommand{\triforceyellow}{
\begin{tikzpicture}[baseline,scale=.5,yshift=-6pt]
\draw[color=YellowOrange, fill=yellow]  (0,0)  -- +(1,0) -- +(60:1) -- cycle;
\draw[color=YellowOrange, fill=yellow]  (1,0)  -- +(1,0) -- +(60:1) -- cycle;
\draw[color=YellowOrange, fill=yellow]  (60:1) -- +(1,0) -- +(60:1) -- cycle;
\end{tikzpicture}
}

\newcommand{\triforce}{
\begin{tikzpicture}[baseline,scale=.2,yshift=-6pt]
\fill[black]
  (0,0)  -- +(1,0) -- +(60:1) -- cycle
  (1,0)  -- +(1,0) -- +(60:1) -- cycle
  (60:1) -- +(1,0) -- +(60:1) -- cycle
;
\end{tikzpicture}
}

\title{Triforce and corners}

\author[Fox]{Jacob Fox}
\address{Department of Mathematics, Stanford University, Stanford, CA 94305, USA}
\email{jacobfox@stanford.edu}
\thanks{JF was
supported by a Packard Fellowship and NSF Career Award DMS-1352121.}

\author[Sah]{Ashwin Sah}
\address{Massachusetts Institute of Technology, Cambridge, MA 02139, USA}
\email{asah@mit.edu}

\author[Sawhney]{Mehtaab Sawhney}
\address{Massachusetts Institute of Technology, Cambridge, MA 02139, USA}
\email{msawhney@mit.edu}

\author[Stoner]{David Stoner}
\address{Harvard University, Cambridge, MA 02138, USA}
\email{dstoner@college.harvard.edu}

\author[Zhao]{Yufei Zhao}
\address{Department of Mathematics, Massachusetts Institute of Technology, Cambridge, MA 02139, USA}
\email{yufeiz@mit.edu}
\thanks{YZ was supported by NSF Award DMS-1764176, and the MIT Solomon Buchsbaum Fund.}

\date{March 2019}

\begin{document}

\begin{abstract}
May the \emph{triforce} be the 3-uniform hypergraph on six vertices with edges $\{123',12'3,1'23\}$. We show that the minimum triforce density in a 3-uniform hypergraph of edge density $\delta$ is $\delta^{4-o(1)}$ but not $O(\delta^4)$.

Let $M(\delta)$ be the maximum number such that the following holds: for every $\epsilon > 0$ and $G = \mathbb{F}_2^n$ with $n$ sufficiently large, if $A \subseteq G \times G$ with $A \ge \delta|G|^2$, then there exists a nonzero ``popular difference'' $d \in G$ such that the number of ``corners'' $(x,y), (x+d,y), (x,y+d) \in A$ is at least $(M(\delta) - \epsilon)|G|^2$. As a corollary via a recent result of Mandache, we conclude that $M(\delta) = \delta^{4-o(1)}$ and $M(\delta) = \omega(\delta^4)$.

On the other hand, for $0 < \delta < 1/2$ and sufficiently large $N$, there exists $A \subseteq [N]^3$ with $|A|\ge\delta N^3$ such that for every $d \ne 0$, the number of corners $(x,y,z), (x+d,y,z),(x,y+d,z),(x,y,z+d) \in A$ is at most $\delta^{c \log (1/\delta)} N^3$. A similar bound holds in higher dimensions, or for any configuration with at least 5 points or affine dimension at least 3. 
\end{abstract}

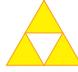
\maketitle

\begin{quote}
    \begin{center} \triforceyellow \end{center}

    \noindent \textit{The Triforce \dots the sacred triangle \dots it is a balance that weighs the three forces: Power, Wisdom and Courage. If the heart of the one who holds the sacred triangle has all three forces in balance, that one will gain the True Force to govern all.}
    
    \hfill --- The Legend of Zelda: Ocarina of Time 
\end{quote}

\section{Introduction}

Green \cite{Green05} ($k=3$) and Green--Tao \cite{GrTa10} ($k=4$) proved the following strengthening of Szemer\'edi's theorem, confirming  a conjecture of Bergelson, Host, and Kra \cite{BHKR}: for every $k \in \{3,4\}$ and $\epsilon >0$ there exists $N_0$ such that if $N \ge N_0$ and $A \subseteq [N]:=\{1,2,\dots,N\}$ has $|A| = \delta N$, then there exists some $d > 0$ such there are at least $(\delta^k - \epsilon)N$ different $k$-term arithmetic progressions with common difference $d$ in $A$, i.e.,
\[
|\{x  : x, x+d, \dots, x+(k-1)d \in A\}| \ge (\delta^k - \epsilon) N.
\]
We abbreviate ``$k$-term arithmetic progression'' as \emph{$k$-AP} from now on. Note that the above estimate is tight for a random subset of $[N]$ of density $\delta$. Even though there exist sets $A$ where the total number of $k$-APs is much less than the random example, the above results say that, for $k \in \{3,4\}$, there always exists some ``popular difference'' $d$ shared by lots of $k$-APs in $A$.

Curiously, the above $k$-AP popular difference result is specific to $k \in \{3,4\}$, as it was shown \cite{BHKR} (with appendix by Ruzsa) that the statement is false for every $k \ge 5$. In particular, it was shown that for each fixed $k\ge 5$, there exists a constant $c_k > 0$ such that for every $0 < \delta < 1/2$, there exists $A \subseteq [N]$ with $|A| \ge \delta N$ such that for every $d > 0$,
\[
|\{x  : x, x+d, \dots, x+(k-1)d \in A\}| \le \delta^{c_k \log(1/\delta)} N.
\]

We are interested in a multidimensional variant of the above result. A \emph{corner} is a pattern of the form $(x,y), (x+d,y), (x,y+d) \in G \times G$ where $G$ is an abelian group and $d \ne 0$ (we can also consider an interval $[N]$ instead of $G$). It is a classic result of Ajtai and Szemer\'edi \cite{AjSz}, known as the \emph{corners theorem}, that every corner-free subset of $G \times G$ has density $o(1)$. Solymosi \cite{Sol} gave a simpler proof of the corners theorem, showing that it follows from the diamond-free lemma.\footnote{The corner theorem was originally stated within the integer grid $[n]^2$, but the graph-theoretic proof extends to $G \times G$ where $G$ is a finite group as observed by Solymosi \cite{Sol13}. The diamond-free lemma is also often stated in equivalent versions known as the induced matching theorem or the $(6,3)$-theorem.}  The diamond-free lemma states that any graph on $n$ vertices in which each edge is in exactly one triangle has $o(n^2)$ edges. The diamond-free lemma in turn follows easily from the triangle removal lemma, that any graph on $n$ vertices with $o(n^3)$ triangles can be made triangle-free by removing $o(n^2)$ edges. 

Given $A \subseteq G \times G$ and $d \in G$, let
\[
S_d(A) = \{(x,y) \in G\times G : (x,y), (x+d,y), (x,y+d) \in A\}.
\]
Then $|S_d(A)|$ is the number of corners in $A$ with common difference $d$.
The naive extension of the popular difference statement is false for corners, as Mandache \cite{Man} showed, improving an earlier construction by Chu \cite{Chu}, that there exists $A \subseteq G \times G$ with $|A| \ge \delta |G|^2$ and $|S_d(A)| = O(\delta^{3.13} |G|^2)$. That is, we cannot always expect to find popular differences shared by as many corners as the random case. Though, perhaps we can hope for a bit less and look for popular differences that are shared by at least $\delta^C |G|^2$ corners---less than random, but still plenty. Note that here the exponent $C$ here should depend only on the pattern and not on the groups. 

Fixing a sequence of abelian groups $G = G_n$ with $|G| \to \infty$, let $M(\delta)$ denote the maximum number such that for every $\epsilon > 0$, if $n$ is sufficiently large and $A \subseteq G \times G$, then there exists $d \ne 0$ such that $|S_d(A)| \ge (M(\delta) - \epsilon)|G|^2$.

Mandache related $M(\delta)$ to the following extremal problem. Let a \emph{triforce}, denoted $\triforce$, be a 3-uniform hypergraphs on 6 vertex with edges $\{123',12'3,1'23\}$. Given a measurable function\footnote{In \cite{Man}, the function $W$ was restricted to piecewise constant functions, but these two formulations are equivalent by a standard approximation argument.} $W \colon [0,1]^3 \to [0,1]$, write
\[
\triforce(W) := \int_{[0,1]^6} W(x',y,z)W(x,y',z)W(x,y,z') \,dxdydzdx'dy'dz'
\]
for the triforce density of $W$.
Let $m(\delta)$ denote the infimum (actually a minimum due to compactness\footnote{The compactness in this case essentially follows from Lov\'asz--Szegedy \cite{LS07} as the triforce is a linear hypergraph. See \cite{ES12,Z15} for general compactness results about hypergraph limits.}) of $\triforce(W)$ among all $W$ with $\int W = \delta$.

\begin{theorem}[Mandache] \label{thm:man}
	For all $G_n$, one has $M(\delta) \le m(\delta)$. Furthermore, if $m_\text{convex} \colon [0,1] \to [0,1]$ is any convex function with $m_{\textrm{convex}} \le m$ pointwise, then for $G_n = \FF_p^n$ with fixed prime $p$, one has $M(\delta) \ge m_{\textrm{convex}}(\delta)$.
\end{theorem}

The first part of \cref{thm:man} was proved via the following randomized construction. For every $a \in G$, let $X_a, Y_a, Z_a$ be i.i.d.\ uniform random variables in $[0,1]$. Given $W$, construct $A \subseteq G \times G$ by including every $(a,b)$ in $A$ with probability $W(X_a, Y_b, Z_{-a-b})$. It is then easy to check that for every $d \ne 0$, $|S_d|/|G|^2$ has expectation $\triforce(W)$ and is concentrated near its mean by Azuma--Hoeffding/McDiarmid's inequality.

The second part of \cref{thm:man} was proved via an arithmetic regularity style argument, combined with ideas from the Fourier analytic proof of the corners theorem. It seems likely that the result can be extended to all abelian groups via Bohr sets, though as far as we know this has not yet been worked out.

Mandache showed that $\delta^4 \le m(\delta) \lesssim \delta^{3.13}$ and asked what is the optimal exponent on $\delta$. Here we show that the answer is $4$, and that the lower bound can be improved by an arbitrarily large factor. 
(We write $f = O(g)$ and $f \lesssim g$ both to mean $f \le Cg$ for some constant $C$, and $f = \omega(g)$ to mean $f/g \to \infty$.)

\begin{theorem} \label{thm:triforce1}
	$m(\delta) = \delta^{4-o(1)}$ and $m(\delta) = \omega(\delta^4)$ as $\delta \to 0$.
\end{theorem}

It is not hard to show that for any $m(\delta)$ with $m(\delta) = \delta^{4-o(1)}$ and $m(\delta) = \omega(\delta^4)$ is lower bounded by  a convex function $m_{\textrm{convex}} \le m$ with $m_{\textrm{convex}}(\delta) = \delta^{4-o(1)}$ and $m_{\textrm{convex}}(\delta) = \omega(\delta^4)$.

\begin{corollary}[Popular corners] \label{cor:corner}
	For all sequence of finite abelian groups $G_n$ with $|G_n| \to \infty$, one has $M(\delta) \le \delta^{4-o(1)}$. Furthermore, for $G_n = \FF_p^n$ with fixed prime $p$, one has $M(\delta) = \omega(\delta^4)$.
\end{corollary}

This result solves Problem 19 from Ben Green's collection of open problems~\cite{GreenOpen}.
For more precise dependencies, the lower bound on $m(\delta)$ is tied to the bound in the triangle removal lemma, and the upper bound on $m(\delta)$ is tied to the bound in the diamond-free lemma, which itself can be bounded using Behrend's construction of 3-AP-free sets~\cite{B}.

The problem of computing $m(\delta)$ is a tripartite analog of the following extremal problem for 3-uniform hypergraphs: determine the minimum density $g(\delta)$ of triforces among 3-uniform hypergraphs with edge density at least $\delta$. Here the \emph{density} of $H$ in $G$ is defined to be the number of homomorphisms from $H$ to $G$ divided by $|V(G)|^{|V(H)|}$. By a standard graph limit argument, this problem is equivalent to the modified extremal problem for $m(\delta)$ where we restrict $W(x,y,z)$ to functions that are symmetric with respect to permutations of its arguments $(x,y,z)$, and as a result, we have $m(\delta) \le g(\delta)$. On the other hand, one has $g(\delta) \lesssim m(\delta)$ by viewing a tripartite 3-uniform hypergraph with $n$ vertices in each part as a 3-uniform hypergraph on $3n$ vertices. Thus, to show \cref{thm:triforce1}, it suffices to show that $g(\delta) = \delta^{4-o(1)}$ and $g(\delta) = \omega(\delta^4)$. In fact, our proof works directly in both settings even without referring to the above relationship between $g$ and $m$.

\begin{theorem} \label{thm:triforce2}
    Let $g(\delta)$ be the minimum triforce density in a 3-uniform hypergraph with edge density at least $\delta$. Then $g(\delta) = \delta^{4-o(1)}$ and $g(\delta) = \omega(\delta^4)$.
\end{theorem}

The natural generalization to $k$-uniform hypergraphs holds with essentially the same proof. We state it explicitly below. Let a \emph{$k$-force} be the $k$-uniform hypergraph with $2k$ vertices and edges $1'2\cdots k$,$12'\cdots k$, \dots, $12\cdots k'$.

\begin{theorem}
\label{thm:kforce}
    Fix $k\ge 3$. Let $g_k(\delta)$ be the the minimum $k$-force density in a $k$-uniform hypergraph with edge density at least $\delta$. Then $g_k(\delta) = \delta^{k+1-o(1)}$ and $g_k(\delta) = \omega(\delta^{k+1})$.
\end{theorem}

One might suspect that the $k$-force density result would lead to consequences for popular $(k-1)$-dimensional corners. However, this is not the case, as our second main result below says.

\begin{theorem} \label{thm:3d-corner}
Let $0 < \delta < 1/2$. For all sufficiently large $N$, there exists $A \subseteq [N]^3$ with $|A| \ge \delta N^3$ such that for all nonzero integer $d$, there are at most $\delta^{c \log (1/\delta)} N^3$ triples $(x,y,z)$ with $(x,y,z), (x+d,y,z),(x,y+d,z),(x,y,z+d) \in A$. Here $c > 0$ is some absolute constant.
\end{theorem}

As we will see in the proof of \cref{thm:3d-corner}, this construction is closely related to the construction in \cite{BHKR} of a set lacking popular differences for 5-APs. We also extend the counterexample of \cite{BHKR} to all $5$-point patterns in $\mathbb{N}$. 

\begin{theorem} \label{thm:5-point}
Let $0 < \delta < 1/2$ and fix five distinct integers $a_1,a_2,a_3,a_4,a_5$. For all sufficiently large $N$, there exists $A \subseteq [N]$ with $|A| \ge \delta N$ such that for all nonzero integer $d$, there are at most $\delta^{c \log (1/\delta)} N$ values of $x \in \ZZ$ such that  $x+a_1d$, $x+a_2d$, $x+a_3d$, $x+a_4d$, $x+a_5d \in A$. Here $c>0$ is a constant depending only on $(a_1,a_2,a_3,a_4,a_5)$.
\end{theorem}

A simple padding and projection of these constructions gives analogous constructions for avoiding popular differences for all patterns with affine dimension at least 3 or at least $5$ points, in particular, $k$-dimensional corners for all $k \ge 3$.

\begin{corollary} \label{cor:3d-pattern}
Fix a finite subset $T \subseteq \ZZ^k$ with affine dimension at least 3 or at least $5$ distinct points. Let $0 < \delta < 1/2$. For all sufficiently large $N$, there exists $A \subseteq [N]^k$ with $|A| \ge \delta N^k$ such that for all nonzero integers $d$, there are at most $\delta^{c \log(1/\delta)} N^k$ points $x \in \ZZ^k$ such that $x + d\cdot T := \{x + dt : t \in T\} \subseteq A$, where $c = c_T > 0$ is a constant.
\end{corollary}

\section{Minimum triforce densities}

Now we begin the proof of \cref{thm:triforce2}. We present the proof in the way that highlights the dependencies between the bounds in Roth's theorem, the diamond-free lemma, the triangle removal lemma, and the minimum triforce density function $g(\delta)$.

Let $\epsilon(\delta)$ be the minimum $\epsilon$ such that every graph on $n$ vertices with at most $\delta n^3$ triangles can be made triangle-free by removing at most $\epsilon n^2$ edges. The triangle removal lemma~\cite{RSz} states that $\epsilon(\delta) \to 0$ as $\delta \to 0$. Note that, by definition, $\epsilon(\delta)$ is a monotonically decreasing function. 

A \emph{diamond-free} graph is a graph in which each edge is in exactly one triangle. Define $d(n)$ so that $d(n)n^2$ is the maximum number of edges of a diamond-free graph on $n$ vertices. The diamond-free lemma states that $d(n)=o(1)$. The diamond-free lemma is an easy consequence of the triangle removal lemma. Indeed, in a diamond-free graph on $n$ vertices with $d(n)n^2$ edges, the number of triangles is $d(n)n^2/3<n^2 = \delta n^3$ with $\delta=1/n$, so we can remove $\epsilon n^2$ edges to make it triangle-free with $\epsilon= \epsilon(1/n)$. However, in such a graph we need to delete an edge from each of these $d(n)/3$ edge-disjoint triangles, so $d(n) \leq 3\epsilon(1/n)$.

Let $r(n)$ be the maximum possible density of a subset of $\ZZ/N\ZZ$ without a $3$-term arithmetic progression. By Roth's theorem~\cite{Roth}, $r(n)=o(1)$, and Behrend's construction~\cite{B} implies that $r(n) \ge e^{-O(\sqrt{\log n})}$. The proof of Roth's theorem from the diamond-free lemma gives a construction of a graph on $3n$ vertices with $3r(n)n^2$ edges, and each edge is in exactly one triangle. Indeed, every 3-AP-free $A \subset \ZZ/N\ZZ$ gives rise to the following diamond-free graph: take the tripartite graph with vertex parts $X, Y, Z$ with $X=Y=Z=\ZZ/N\ZZ$ and edges $(x,x+a) \in X\times Y$, $(y,y+a) \in Y \times Z$ and $(x,x+2a) \in X \times Z$, ranging over all $a \in A,x \in X, y \in Y, z \in Z$. Thus, $3r(n)n^2 \leq d(3n)(3n)^2$, or equivalently, $d(3n) \geq r(n)/3 \ge e^{-O(\sqrt{\log n})}$.

The next lemma implies the claimed upper bound on $g(\delta)$. To avoid confusion, we use ``triple'' to refer to an edge of a 3-uniform hypergraph, ``edge'' to refer to an edge of a graph, and ``triangle'' to refer to a triangle in a graph.

\begin{lemma}\label{lem:g-lowerbd}
For every integer $n \geq 3$ and $\delta=2d(n)/n$, we have $\displaystyle g(\delta) \leq \frac{\delta^4}{8d(n)^3}$. 
\end{lemma} 

\begin{proof}
Let $G$ be a graph with $n$ vertices and $d(n)n^2$ edges such that every edge is in exactly one triangle. Let $H$ be the $3$-uniform hypergraph on the same vertex set as $G$ in which the triples of $H$ are precisely the triangles of $G$. The triple density of $H$ is $\delta=2d(n)/n$. The only homomorphisms from the triforce to $H$ map the vertices of the triforce to those of a single triple of $H$, so the triforce density is $2d(n)n^2/n^6=\frac{\delta^4}{8d(n)^3}$. 
\end{proof}

From the above lemma and the Behrend bound described above, we get $g(\delta) \leq \delta^{4-O(1/\sqrt{\log (1/\delta)})}$. In particular, $g(\delta)=\delta^{4-o(1)}$.  

The next lemma implies the claimed lower bound on $g(\delta)$.

\begin{lemma}\label{lem:g-upperbd}
For every $\delta > 0$, we have $\displaystyle g(6\delta) \geq \frac{\delta^4}{3\epsilon\left(\delta/\epsilon(\delta)\right)}$. 
\end{lemma} 
\begin{proof}
Let $H$ be a $3$-uniform hypergraph with at least $\delta n^3$ triples (i.e., triple density at least $6\delta$). If $H$ has a pair of vertices in at most $\delta n$ triples, then delete all triples of $H$ containing that pair of vertices, and repeat. In total, we delete at most $\binom{n}{2} \delta n$ triples, so the remaining subhypergraph $H'$ has at least $(\delta/2) n^3$ triples, and each pair of vertices lying in some triple of $H'$ actually lies in at least $\delta n$ triples of $H'$. 

Let $G$ be the graph with the same vertex set as $H$, where $uv \in E(G)$ if $u$ and $v$ both lie in some triple of $H'$ (and hence in at least $\delta n$ triples of $H'$). We can extend each triangle in $G$ to at least $6(\delta n)^3$ triforces (as homomorphisms). Hence, if $G$ has at least $\delta n^3/\epsilon(\delta)$ triangles, then $H'$ (and hence $H$) has at least $6(\delta n)^3 \cdot \delta n^3/\epsilon(\delta)=6 \delta^4 n^6 / \epsilon(\delta)$ triforces, and we would be done. 

We have reduced the problem to the case where $G$ has fewer than $\delta n^3 / \epsilon(\delta)$ triangles. By the triangle removal lemma, there is an edge subset $S \subseteq E(G)$ such that $|S| \leq \epsilon(\delta/\epsilon(\delta))n^2$ and every triangle of $G$ contains at least one edge in $S$. Note that every triple of $H'$ forms a triangle in $G$, so every triple of $H'$ contains at least one edge in $S$. Call an edge $uv \in S$ \emph{strong} if $u$ and $v$ are contained in at least $\frac{\delta}{4|S|}n^3$ triples of $H'$. The number of triples of $H'$ that contain a strong edge is at least $e(H')-|S| \cdot \frac{\delta}{4|S|}n^3 \geq \frac{\delta}{4}n^3$. Every triple of $H'$ that contains a strong edge can be extended  to at least $6(\delta n)^2 \cdot \frac{\delta}{4|S|}n^3 \geq \frac{3\delta^3}{2\epsilon(\delta/\epsilon(\delta))}n^3$ triforces. Thus, the number of triforces is at least $$\frac{\delta}{4}n^3 \cdot \frac{3\delta^3}{2\epsilon(\delta/\epsilon(\delta))}n^3 > \frac{\delta^4}{3\epsilon(\delta/\epsilon(\delta))}n^6.$$ 
\end{proof}

As $\epsilon(\delta)=\delta^{o(1)}$ (this follows from the Behrend bound as described above), and $\epsilon(\delta) \to 0$ as $\delta \to 0$, \cref{lem:g-upperbd} implies that $g(\delta)=\omega(\delta^4)$.

\cref{lem:g-upperbd,lem:g-lowerbd} complete the proof of \cref{thm:triforce2}.

\medskip

The proof of \cref{thm:kforce} is almost verbatim the same as that of \cref{thm:triforce2}. In modifying the above proofs, $H$ becomes a $k$-uniform hypergraph and $G$ becomes a $(k-1)$-uniform hypergraph. Instead of the triangle removal lemma, we would apply the simplex removal lemma for hypergraphs. We omit the details.

\section{Constructions avoiding popular differences}\label{sec:constructions}

In this section we prove \cref{thm:3d-corner,thm:5-point}.

\subsection{Motivation of proof}

Before diving into the proof it is worth considering the relationship between this result and the counterexample given by Bergelson, Host, Kra, and Ruzsa \cite{BHKR} for  popular differences for $5$-AP's. The counterexample for 5-AP popular difference in \cite{BHKR} essentially\footnote{Note that in \cite{BHKR} the counterexample is only stated for infinitely many $N$. We require additional considerations in order for the result to apply for all sufficiently large $N$ and these can be adapted to the setting in \cite{BHKR}. Conversely the number theoretic considerations in our proof are simplified when giving a counterexample for only infinitely many $N$.} implies the analogous version of \cref{thm:3d-corner} for 4-dimensional corners by pulling back the 5-AP counterexample via the map $(x,y,z,w)\to x+2y+3z+4w$. 

This pull-back construction and the Green-Tao \cite{GrTa10} result on popular differences for $4$-APs may suggest that the popular difference result for $2$-dimensional corners extends to $3$-dimensional corners. However, the Green-Tao \cite{GrTa10} proof (based on \cite{BHKR}) ultimately boils down to the identity $P(0) - 3 P(1) + 3 P(2) - P(3) = 0$ for quadratic polynomials $P$. This identity  yields a critical positivity as it can be rewritten as $P(0) - 3P(1) = P(3)- 3 P(2)$ and there is an inherent symmetry between the left and right hand sides. However, a $3$-dimensional corner can project not only to a $4$-AP but other $4$-term progressions, e.g., $x, x+y, x+2y, x+4y$ and such patterns do not posses the same ``magical'' positivity.

Ruzsa's construction relies on the property that if $P$ is a nonconstant univariate quadratic polynomial taking $\RR^n$ values, then $P(0), P(1), P(2), P(3), P(4)$ cannot all lie on the unit sphere. This can be proved using the identities $P(0) - 3 P(1) + 3 P(2) - P(3) = 0$ and $P(1) - 3 P(2) + 3P(3) - P(4) = 0$, and observing that no 5 points on a sphere can satisfy these linear relations.

If there were an identity expressing $P(3)$, say, as a convex combination of $P(0)$, $P(2)$, and $P(6)$, over all quadratic polynomials $P$, then one can mimic the proof given in \cite{BHKR} to construct a subset of $[N]^3$ without popular 3-dimensional corners. However, this approach fails as no such identity can exist, since if $a_1 P(b_1) + \cdots + a_4 P(b_4) = 0$ for all quadratic polynomials $P$, where $a_i$, $b_i$'s are constants with $a_1+a_2+a_3+a_4=0$, all $b_i$'s distinct, and $a_i$'s not all zero, then by Lagrange interpolation one can deduce that there must be two positive and two negative values among the $a_i$'s. However, we do not know if there are subsets of $[N]$, in the style of Behrend's construction, with $N^{1-o(1)}$ elements avoiding such patterns (e.g., $2x + 2y=3w + z$).

However, the above discussion is limited to univariate polynomials. We can circumvent this difficulty by constructing a quadratic polynomial $f(x,y,z)$ satisfying $f(x+d,y,z) + f(x,y+d,z) + f(x,y,z+d) = 3 f(x,y,z)$ and then carry out the strategy in \cite{BHKR}. 

\subsection{Proof of \cref{thm:3d-corner}}
Define 
\begin{equation}\label{eq:f}
f(x, y, z) =  (x - y) (x + y - 2 z),
\end{equation}
which satisfies the following useful identity:
\begin{equation}\label{eq:f-id}
    f(x+d,y,z)+f(x,y+d,z)+f(x,y,z+d)=3f(x,y,z).
\end{equation}

\begin{lemma}\label{lem:behrend-modified}
There is an absolute constant $c > 0$ such that the following holds. For every integer $L > 0$ there exists a subset $\Lambda$ of $\{0, 1, \ldots, L - 1\}$ having at least $L\exp(-c\sqrt{\log L})$ elements that does not contain any nontrivial solutions to $x+y+z=3w$ (here a trivial solution is one with $x=y=z=w$).
\end{lemma}

\begin{proof}
This follows from a standard modification from  Behrend's construction~\cite{B} of a large 3-AP-free set (e.g., see \cite[Lemma 3.1]{Alon}).
\end{proof}

For the equation $x+y+z = 3w$ this construction is known be essentially the best possible due to the work of Schoen and Sisask \cite{SS}. The next lemma is similar to Lemma 2.3 in \cite{BHKR}.

\begin{lemma}\label{lem:alpha-transfer}
Let $\Lambda$ be a subset of $\{0, 1, \ldots, L - 1\}$ not containing any nontrivial solutions to $a+b+c=3d$ and $\alpha$ be a fixed real constant. For each $j\in\Lambda$, let
\[I_j := \left[\frac{j}{3L}, \frac{j}{3L} + \frac{1}{9L}\right) \subseteq \TT := \RR/\ZZ,
\]
and let
\[
B = \bigcup_{j \in \Lambda} I_j.
\]
Let $f$ be the polynomial in \eqref{eq:f}. Let $w = \alpha f(n_1,n_2,n_3)$, $x = \alpha f(n_1 + d,n_2,n_3)$, $y = \alpha f(n_1,n_2 + d,n_3)$, and $z = \alpha f(n_1,n_2,n_3+d)$ and suppose that $w, x, y, z \pmod 1$ all lie in $B$. Then 
\[
\norm{2\alpha (n_1-n_2)d}_{\RR/\ZZ} < \frac{1}{9L},
\]
where $\norm{x}_{\RR/\ZZ}$ denotes the distance from $x \in \RR$ to the closest integer.
\end{lemma}

\begin{proof}
By \eqref{eq:f-id},
we have $x+y+z=3w$. Let $W, X, Y, Z\in\Lambda$ be such that $w\in I_{W}$, $x\in I_{X}$, $y\in I_{Y}$, and $z\in I_{Z}$. Then
$
x+y+z \pmod{1}$ lies in $[\frac{X+Y+Z}{3L}, \frac{X+Y+Z}{3L} + \frac{1}{3L})$ and $3w \pmod{1}$ lies in $[\frac{3W}{3L}, \frac{3W}{3L} + \frac{1}{3L})$.
Since $X+Y+Z < 3L$, these two intervals intersects exactly when $W+X+Y = 3Z$, which implies that $W=X=Y=Z$ since $\Lambda$ has no nontrivial solutions to this equation. The conclusion follows from the identity $2\alpha d(n_1-n_2) = w - z$ and that $w$ and $z$ both lie in the interval $I_W$ with length $1/(9L)$.
\end{proof}

Finally, as in \cite{BHKR}, we need irrational numbers well-approximable by fractions with a special property (\cite{BHKR} only contains a sketch of this part of the argument). In \cite{BHKR} the result for 5-APs is proved only for infinitely many values of $N$. In order to make the construction work for all sufficiently large $N$, we need to construct a set of such irrational numbers rather than the single number $\alpha$ used in \cite{BHKR}, which explains some of the technicalities to follow. Later, we explain how to extend the 5-AP construction for infinitely many $N$ in \cite{BHKR} to all sufficiently large $N$ for any fixed five point pattern.

Before proceeding with the proof of the next lemma, we require several elementary facts regarding continued fractions. Here we use the standard notation that 
\[(c_0; c_1, c_2, \ldots) := c_0+\frac{1}{c_1+\frac{1}{c_2+\ldots}}
\quad\text{and}\quad
(c_0; c_1, c_2, \ldots,c_k) := c_0+\frac{1}{c_1+\frac{1}{c_2+\ldots\frac{1}{c_k}}}.
\]
Let the \emph{$k$-th approximant} of the continued fraction be
\[
(c_0; c_1, c_2, \ldots,c_{k}) = P_k/Q_k,
\]
where $P_k$ and $Q_k$ are relatively prime positive integers. By standard facts about continued fractions, we have the double recurrence
\begin{equation}\label{eq:continued-fraction-1}
P_k = c_k P_{k-1}+P_{k-2} \quad \text{and} \quad Q_k = c_k Q_{k-1} + Q_{k-2}
\end{equation}
for $k\ge 0$ with $Q_{-1} = 0$ and $Q_{-2} = 1$. Finally if $\alpha = (c_0; c_1, c_2, \ldots)$ and $n\ge 1$, then
\begin{equation}\label{eq:continued-fraction-2}
\left|\alpha - \frac{P_n}{Q_n}\right| < \frac{1}{Q_nQ_{n + 1}}.
\end{equation}
\begin{lemma}\label{lem:alpha-hard}
Fix a positive integer $m > 1$. Then there is a real $b \in (1, 2^{2m + 1}]$ such that the following holds. For all real $r > 0$, there is an irrational number $\alpha$ and infinitely many fractions $p_i/q_i$ with relatively prime positive integers $p_i < q_i$ and $q_i$ having no prime factor smaller than $m$ such that $|\alpha - p_i/q_i| < 1/(mq_i^2)$, and $rb^i < q_i < 2rb^i$ for $i\ge i(r, m, b)$ sufficiently large.
\end{lemma}

\begin{proof}
In order to construct the desired irrational $\alpha$ we build its continued fraction expansion $(c_0;c_1,c_2, \ldots )$ iteratively. Let $a = \text{lcm}(1, 2, \ldots, m)$ and set $b = \frac{a + \sqrt{a^2 + 4}}{2}$. Then  $a<4^m$ (see \cite{Nair}) and thus $2\le b\le 2^{2m + 1}$. Choose $K$ such that $rb^K > 2m$. It follows from Bertrand's postulate that there exist primes $x, y$ such that $x\in (rb^N, 2rb^N)$ and $y\in (rb^{N+1},2rb^{N+1})$. Since $2m < x < y$ and $x$ and $y$ are primes, we have $\gcd(xy, a) = 1$ and $\gcd(x, y) = 1$.

Now we check that there exist sequences $\{c_i\}_{i\ge 0}, \{P_i\}_{i\ge 0}, \{Q_i\}_{i\ge 0}$ of positive integers with the following properties:
\begin{itemize}
\item $P_k/Q_k=(c_0;c_1,c_2,\ldots, c_k)$ for each $k\ge 1$, where $P_k$ and $Q_k$ are relatively prime, and
\item there exists a positive integer $t$ with $Q_t=x, Q_{t+1}=y$.
\end{itemize}
Indeed, this follows from running the Euclidean algorithm on $x,y$ and using the recurrence relation \eqref{eq:continued-fraction-1}. This establishes the value of $c_i$ for all $i\le t+1$. Set $c_i = a$ for all $i\ge t + 2$, and set $\alpha=(c_0;c_1,c_2,\ldots)$ Now since $Q_t = x\in (rb^K,2rb^K)$ and $Q_{t + 1} = y\in (rb^{K+1},2rb^{K+1}),$ it inductively follows that $Q_n\in (rb^{n-t+K},2rb^{n-t+K})$, since $Q_n = aQ_{n - 1} + Q_{n - 2}$ for $i\ge t + 2$ and $b$ satisfies $b^2 = ab + 1$.
Furthermore, by construction, we have $Q_n\equiv Q_{n - 2}\pmod{a}$ for $n\ge t + 2$. Thus either $Q_n\equiv x\pmod{a}$ or $Q_n=y\pmod{a}$ for all $n\ge t$, each of which is relatively prime to $a$. Thus $Q_n$ for $n\ge t$ is relatively prime to $a$, and since $a = \text{lcm}(1, 2, \ldots, m)$, we see that $Q_n$ for $n\ge t$ has no prime factors of size at most $m$. We can now let $(p_i, q_i) = (P_{i + t - K}, Q_{i + t - K})$ for sufficiently large $i$, and by construction these integers satisfy the required conditions. Finally, by \eqref{eq:continued-fraction-2} we have
\[\left|\alpha - \frac{P_n}{Q_n}\right| < \frac{1}{Q_nQ_{n + 1}} < \frac{1}{aQ_n^2}\le\frac{1}{mQ_n^2}.\]
\end{proof}
We are ready to prove \cref{thm:3d-corner}.
\begin{proof}[Proof of \cref{thm:3d-corner}]
We may assume that $\delta$ is sufficiently small or otherwise we can take $A = [N/2]\times [N]\times [N]$ and then the theorem is true if the constant is chosen appropriately.

Let $L = \exp(c \log(1/\delta)^2)$ for an appropriately chosen sufficiently small constant $c > 0$. Apply \cref{lem:alpha-hard} for $m = L$ and $t = 2L + 1$ different values of $r$, namely $r = 2^j$ for $1\le j\le 2L + 1$. The lemma gives a single $b \in (1, 2^{2L + 1}]$ and irrationals $\alpha_1, \ldots, \alpha_t$ as well as positive integers $p_{j, i}, q_{j, i}$ with $\gcd(p_{j, i}, q_{j, i}) = 1$ so that for all $j \in [t]$,
\begin{itemize}
    \item $q_{j, i}\in (2^jb^i, 2^{j + 1}b^i)$ for sufficiently large $i\ge i(j)$, and
    \item $\gcd(q_{j, i}, \text{lcm}(1, \ldots, L)) = 1$ for $i\ge i(j)$, and
    \item $|\alpha_j - \frac{p_{j, i}}{q_{j, i}}| < \frac{1}{Lq_{j, i}^2}$ for $i\ge i(j)$.
\end{itemize}
Let $I = \max\{i(1), \ldots, i(t)\}$. Then the above properties hold for all $1\le j\le t$ and $i\ge I$. Observe that all sufficiently large $N$ (here ``sufficiently large'' depends on $\delta$) are within a factor of $4$ from some $q_{j, i}$ with $1\le j\le t$ and $i\ge I$. Therefore, to prove the theorem for all sufficiently large integers $N$, it suffices to prove it for numbers of the form $N = q_{j, i}$.

Let $N = q_{j, i}$ with $1\le j\le t$ and $i\ge I$. Let $\alpha = \alpha_j$. Define
\[F = \{n\in\mathbb{N}: n\alpha\in B\pmod{1}\},\]
where, as in \cref{lem:alpha-transfer},
\[
B = \bigcup_{k \in \Lambda} \left[\frac{k}{3L}, \frac{k}{3L} + \frac{1}{9L}\right) \subseteq \TT 
\]
and $\Lambda$ is a subset of $\{0, 1,\dots, L-1\}$ of size $Le^{- O(\sqrt{\log L})}$ not containing nontrivial solutions to $a+b+c=3d$ (by \cref{lem:behrend-modified}). Let $f(x, y, z) =  (x - y) (x + y - 2 z)
$ as in \eqref{eq:f}, and
\begin{equation}
A = \{(x_1, x_2, x_3)\in [N]^3: f(x_1, x_2, x_3)  \in F\}.
	\label{eq:construction-A}
\end{equation}
By the Weyl equidistribution criterion (e.g., see \cite{TaoAnalysis}), using $m(\cdot)$ for Lebesgue measure, as $N \to \infty$,
\[
\frac{|A|}{N^3} \to m(B) = \frac{|\Lambda|}{9L} = e^{-O(\sqrt{\log L})} \ge 2\delta 
\] 
as long as we have chosen the constant $c$ in $L = \exp(c \log(1/\delta)^2)$ so that the last inequality is true. Thus, for sufficiently large $N$, we have $|A|\ge\delta N^3$.

A key point here is that while the rate of convergence of the equidistribution claim may depend on $\alpha$, since there are only finitely many $\alpha$'s that we need to consider, there is a single $N_0(\delta)$ such that $|A| \ge \delta N^3$ whenever $N = q_{j,i} \ge N_0(\delta)$ with $j \in [t]$ and $i \ge I$ as above.

Fix a nonzero integer $s$ with $|s| < N$, which will be the common difference of the corners that we are counting. Suppose $(a_1, a_2, a_3)\in A$ generates a corner of common difference $s$ within $A$, i.e. 
\[
(a_1, a_2, a_3), (a_1 + s, a_2, a_3), (a_1, a_2 + s, a_3), (a_1, a_2, a_3 + s) \in A.
\]
Then 
\[
f(a_1, a_2, a_3), f(a_1 + s, a_2, a_3), f(a_1, a_2 + s, a_3), f(a_1, a_2, a_3 + s) \in F
\]
by the construction~\eqref{eq:construction-A}. By \cref{lem:alpha-transfer}, $\norm{2s(a_1 - a_2)\alpha}_{\RR/\ZZ} < 1/(9L)$. So
\begin{align*}
\norm{2s(a_2 - a_3)\frac{p_{j, i}}{q_{j, i}}}_{\RR/\ZZ}
&\le \norm{2s(a_1 - a_2)\alpha}_{\RR/\ZZ} + 
\abs{2s(a_1 - a_2)\alpha - 2s(a_1 - a_2)\frac{p_{j, i}}{q_{j, i}}} 
\\
&\le \frac{1}{9L} + 2s\abs{a_1 - a_2}\abs{\alpha - \frac{p_{j, i}}{q_{j, i}}} 
\\
&\le \frac{1}{9L} + 2N^2 \cdot \frac{1}{Lq_{j, i}^2} = \frac{1}{9L} + \frac{2}{L} \le \frac{3}{L}
\end{align*}

Recall that $N = q_{j,i}$ is relatively prime to all of $[L]$ as well as to $p_{j,i}$. In particular, $N$ is odd. Also $|s| < N$, so $s$ is not divisible by $N$. It follows that $2sp_{j, i}/q_{j, i}$ is not an integer. Writing $2sp_{j, i}/q_{j, i} = P/Q$ where $P$ and $Q$ are relatively prime integers with $Q$ positive, one has $Q > L$ since all prime divisors of $q_{j,i}$ are greater than $L$.

Thus $\norm{(a_2-a_3) P/Q}_{\RR/\ZZ} \le 3/L$. So $(a_2-a_3)P \pmod Q \in [-\floor{3Q/L},\floor{3Q/L}]$. Since multiplication by $P$ is a bijection in $\ZZ/Q\ZZ$, there are at most $1 + 6Q/L \le 7Q/L$ possible values that $a_2 -a_3$ can take in $\ZZ/Q\ZZ$, and hence there are $7 N/L$ possible values (recall $N/Q \in \ZZ$) that $a_2 - a_3$ can take in $[0,N)$. This gives at most $14 N/L$ possible values $a_2 - a_3$ can take in $(-N, N)$. Therefore there are at most $14 N^3/L \le  14e^{- c\log(1/\delta)^2}N^3$ different points $(a_1, a_2, a_3) \in [N]^3$ that generates a corner of common difference $s$.
\end{proof}

\subsection{Proof of \cref{thm:5-point}} The construction for 5-point patterns is a modification of Ruzsa's construction for 5-APs in \cite{BHKR}.

\begin{definition}
	Let $\bm a = (a_1, \dots, a_k) \in \ZZ^k$ be a vector of distinct integer coordinates.
	A \emph{quadratic configuration of type $\bm a$}, abbreviated $QC(\bm a)$, is a vector of the form $(P(a_1), \dots, P(a_k)) \in \ZZ^k$ where $P$ is some non-constant polynomial of degree at most 2. 
	
	We say that some set of integers $S$ contains a $QC(\bm a)$ if $P(a_1), \dots, P(a_k) \in S$ for some nonconstant polynomial $P$ of degree at most $2$.
\end{definition}

\begin{remark}
For 5-APs, \cite{BHKR} only considers $\bm a = (0,1,2,3,4)$, in which case $QC(\bm a)$ was called a \emph{QC5}.
\end{remark}

\begin{lemma} \label{lem:qc}
	Let $k \ge 4$. Given $\bm a = (a_1, \dots, a_k) \in \ZZ^k$ with distinct coordinates, there exists a  vector $(\gamma_{i,0}, \dots, \gamma_{i,3}) \in (\ZZ\setminus\{0\})^4$ for each $1 \le i \le k-3$ such that $\bm y = (y_1, \dots, y_k) \in \ZZ^k$ is a $QC(\bm a)$ if and only if not all entries of $\bm y$ are equal and 
	\begin{equation}
		\label{eq:qc-linear}
		\gamma_{i,0} y_i +  \gamma_{i,1} y_{i+1} + \gamma_{i,2} y_{i+2} + \gamma_{i,3} y_{i+3}  = 0 \quad  \text{ for each } 1 \le i \le k-3.
	\end{equation}
\end{lemma}

\begin{proof}
	For each $1 \le i \le k-3$ and $0 \le j \le 3$, set
	\[
	\gamma_{i,j} = M \prod_{s \in \{0,1,2,3\} \setminus \{j\}} (a_{i+j} - a_{i+s})^{-1}
	\]
	where $M$ is some positive integer so that all $\gamma_{i,j}$'s are integers.
	We have that for all polynomials $P$ of degree at most $2$, and all $1 \le i \le k-3$,
	\begin{equation}
		\label{eq:qc-P}
		\gamma_{i,0} P(a_i) +  \gamma_{i,1} P(a_{i+1}) + \gamma_{i,2} P(a_{i+2}) + \gamma_{i,3} P(a_{i+3})  = 0 \quad  \text{ for each } 1 \le i \le k-3.
	\end{equation}
	The above identity is essentially the Lagrange polynomial interpolation formula. It can also be verified, for each $i$, by checking the identity on three linearly independent polynomials $P(x) = \prod_{s \in \{0,1,2\} \setminus \{j\}} (x - a_{i+s})$ for $j \in \{0,1,2\}$. Both the linear independence claim and the quadratic polynomial identity can be verified by evaluations at $x = a_i, a_{i+1}, a_{i+2}$.

	If $\bm y$ is $QC(\bm a)$, then by definition $y_i = P(a_i)$ for some nonconstant polynomial $P$ of degree at most $2$, and thus $\bm y$ satisfies \eqref{eq:qc-linear}. Furthermore, $P$ cannot take the same value more than twice since it has degree at most 2, so not all coordinates of $\bm y$ are equal.
	
	Conversely, suppose some non-constant vector $\bm y$ satisfies \eqref{eq:qc-linear}. Let $P$ be a polynomial of degree at most 2 such that $P(a_j) = y_j$ for $j = 1,2,3$. Comparing \eqref{eq:qc-linear} and \eqref{eq:qc-P} for each $i = 1, 2, \dots, k-3$ sequentially, noting that $\gamma_{i,j} \ne 0$ for all $i,j$, we find that $P(a_j) = y_j$ for all $1 \le j \le k$. Since $\bm y$ is non-constant, $P$ is non-constant as well.
\end{proof}

%

The next lemma is a modification of Behrend's construction, following Ruzsa's appendix \cite{BHKR}.

\begin{lemma}\label{lem:behrend-modified-2}
Fix a vector $\bm a = (a_1, \dots, a_5)$ of five distinct integers. There exists a constant $C = C_{\bm a}$ such that for every positive integer $L$, there exists $\Lambda \subseteq \{0,1, \dots, L-1\}$ with $|\Lambda| \ge Le^{-C\sqrt{\log L}}$ that does not contain any $QC(\bm a)$.
\end{lemma}

\begin{proof}
Let $\gamma_{i,j}$, $i \in \{1,2\}$, $j \in \{0,1,2,3\}$ be the integer coefficients from \cref{lem:qc}. Let $\Gamma = 4 \max_{i,j} |\gamma_{i,j}|$. Define
\[
\Lambda = \{x_0 + x_1 m + \cdots + x_{d-1} m^{d-1} : \text{each }x_j \in \{0, 1, \dots, \floor{m/\Gamma}-1\} \text{ and } \ \sum_{j=0}^{d-1} x_j^2 = r\}
\]
where $d = \floor{\sqrt{\log L}}$, $m = \floor{L^{1/d}}$, and $r < d(m/\Gamma)^2$ chosen so that $\abs{\Lambda}$ is as large as possible. Then $|\Lambda| \ge \floor{m/\Gamma}^d/(d(m/\Gamma)^{2}) \ge L e^{-C\sqrt{\log L}}$.

It remains to show that $\Lambda$ contains no $QC(\bm a)$. Indeed, suppose $y_1, \dots, y_5 \in \Lambda$ form a $QC(\bm a)$ via non-constant polynomial $P$ of degree at most $2$ so that $y_i = P(a_i)$ for each $i \in [5]$. For each $i \in [5]$, let
\[
y_i = x_{i,0} + x_{i,1} m + \cdots + x_{i,d-1} m^{d-1}
\]
with integers $x_{j,i} \in \{0,1, \dots, \floor{m/\Gamma}-1\}$. By \eqref{eq:qc-linear}, we have, for each $i = 1,2$,
\[
	\sum_{j=0}^{d-1} (\gamma_{i,0} x_{i,j} +  \gamma_{i,1} x_{i+1,j} + \gamma_{i,2} x_{i+2,j} + \gamma_{i,3} x_{i+3,j}) m^j = 0.
\]
The coefficient of each $m^j$ is an integer less than $m$ in absolute value, so they must all be zero due to the uniqueness of base-$m$ expansion. It follows that 
\[
\gamma_{i,0} x_{i,j} +  \gamma_{i,1} x_{i+1,j} + \gamma_{i,2} x_{i+2,j} + \gamma_{i,3} x_{i+3,j} = 0
\]
for each $i \in \{1,2\}$ and $j \in \{0, \dots, d-1\}$. Then, for each $j \in \{0, \dots, d-1\}$ there exist a polynomial $P_j$ of degree at most $2$ such that $P_j(a_i) = x_{i,j}$ for each $i \in [5]$. Let $\bm P(t) = (P_0(t), \dots, P_{d-1}(t))$. Then $|\bm P(a_i)|^2 = r$ for each $i \in [5]$ by the construction of $\Lambda$. So $|\bm P(t)|^2 - r$ is a polynomial of degree at most 4 with 5 distinct real roots $t = a_1, \dots, a_5$, and thus $\bm P(t)$ must be a constant, and so $y_1 = \cdots = y_5$. Hence $\Lambda$ has no $QC(\bm a)$.
\end{proof}

The next lemma is analogous to \cref{lem:alpha-transfer}.

\begin{lemma}\label{lem:alpha-transfer-2}
Fix a vector $\bm a = (a_1, \dots, a_5)$ of five distinct integers. There exist positive integers $\Theta_1, \Theta_2, \Theta_3$ depending only on $\bm a$ such that the following holds. Let $\Lambda$ be a subset of $\{0, 1, \ldots, L - 1\}$ not containing any $QC(\bm a)$. For each $j\in\Lambda$, let
\[I_j := \left[\frac{j}{\Theta_1 L}, \frac{j}{\Theta_1 L} + \frac{1}{\Theta_1^2 L}\right) \subseteq \TT := \RR/\ZZ,
\]
and let
\[
B = \bigcup_{j \in \Lambda} I_j.
\]
Let $\alpha \in \RR$, $n,d \in \ZZ$, and $u_i = \alpha (n + a_id)^2$ for $1 \le i \le 5$, and suppose that $u_i \pmod 1$ lies in $B$ for each $i \in [5]$. Then $\norm{\Theta_2\alpha nd}_{\RR/\ZZ} < \Theta_3/L$.
\end{lemma}
\begin{proof}
Let $\gamma_{i,j}$, $i \in \{1,2\}$, $j \in \{0,1,2,3\}$, be the integer coefficients from \cref{lem:qc}. Let $\Theta_1 = 4\max_{i,j} |\gamma_{i,j}|$. Let $U_1, \dots, U_5 \in \Lambda$ such that $u_i \pmod 1 \in I_{U_i}$ for each $i \in [5]$. Since the linear relations \eqref{eq:qc-P} holds for the quadratic polynomial $P(t) = \alpha(n+ dt)^2$, we have
\[
	\gamma_{1,0} u_1 +  \gamma_{1,1} u_{2} + \gamma_{1,2} u_{3} + \gamma_{1,3} u_{4}  = 0.
\]
For each $i$, we have $u_i \pmod 1 \in I_{U_i}$, and so $\norm{u_i - U_i/(\Theta_1L)} < 1/(\Theta_1^2L)$. Considering the above displayed equality, we have
\[
\norm{\frac{\gamma_{1,0} U_1 + \gamma_{1,1} U_2 + \gamma_{1,2} U_3 + \gamma_{1,3} U_4}{\Theta_1 L}}_{\RR/\ZZ} < \frac{\Theta_1}{\Theta_1^2L} = \frac{1}{\Theta_1L}.
\]
Since $|\gamma_1 U_1 + \gamma_2  U_2 + \gamma_3  U_3 + \gamma_4  U_4| < \Theta_1 L$, it follows that
\[
\gamma_{1,0} U_1 + \gamma_{1,1} U_2 + \gamma_{1,2} U_3 + \gamma_{1,3} U_4 = 0
\] 
as integers. Likewise,
\[
\gamma_{2,0} U_2 + \gamma_{2,1} U_3 + \gamma_{2,2} U_4 + \gamma_{2,3} U_5 = 0.
\]
Since $U_1, \dots, U_5 \in \Lambda$ and $\Lambda$ has no $QC(\bm a)$, we must have $U_1 = \cdots = U_5$ due to the characterization of $QC(\bm a)$ in \cref{lem:qc}.

We have the identity
\[
2(a_1 - a_2)(a_2 - a_3)(a_3 - a_1) n d = (a_2^2 - a_3^2) (n + da_1)^2 + (a_3^2 - a_1^2) (n + da_2)^2 + (a_1^2 - a_2^2) (n + da_3)^2.
\] 
Thus, setting $\Theta_2 = |2(a_1 - a_2)(a_2 - a_3)(a_3 - a_1)|$ and $\Theta_3 = \ceil{3 \max_{i\in[3]} |a_i|^2/\Theta_1^2}$, we have
\begin{align*}
&\norm{\Theta_2\alpha nd}_{\RR/\ZZ}
\\
&= \norm{(a_2^2 - a_3^2)u_1 + (a_3^2 - a_1^2)u_2 + (a_1^2 - a_2^2)u_3}_{\RR/\ZZ}
\\
&\le  \norm{(a_2^2 - a_3^2)\frac{U_1}{\Theta_1 L} + (a_3^2 - a_1^2)\frac{U_2}{\Theta_1 L} + (a_1^2 - a_2^2)\frac{U_3}{\Theta_1 L}}_{\RR/\ZZ}
+ (3 \max_{i\in[3]} |a_i|^2) \max_{i \in [3]} \norm{u_i - \frac{U_i}{\Theta_1 L}}_{\RR/\ZZ}
\\
&=  (3 \max_{i\in[3]} |a_i|^2) \max_{i \in [3]} \norm{u_i - \frac{U_i}{\Theta_1 L}}_{\RR/\ZZ} < \frac{\Theta_3}{L}, \end{align*}
In the last equality step we are using that $U_1 =U_2 = U_3$.
\end{proof}

\begin{proof}[Proof of \cref{thm:5-point}]
The construction is nearly identical to the proof of \cref{thm:3d-corner}, except that now we apply \cref{lem:alpha-transfer-2} instead of \cref{lem:alpha-transfer}. In particular, let $\Lambda$ be as in \cref{lem:behrend-modified-2}, construct $B$ as in \cref{lem:alpha-transfer-2}, and let $F = \{n\in\NN: n\alpha\in B\pmod{1}\}$ where $\alpha$ is a suitable irrational number (as in the proof of \cref{thm:3d-corner}). Then $A = \{x\in [N]: x^2\in F\}$ is the desired set.
\end{proof}

\subsection{Proof of \cref{cor:3d-pattern}} We first consider the case when the pattern $T$ has at least $5$ points. By removing extra points, it suffices to prove the result with $|T| = 5$. Fix an integer $C$ larger than the sum of the magnitudes of the coordinates of all points in $T$. 
Define the map $\varphi: (x_1,\ldots,x_k)\mapsto \sum _{i=1}^{k} C^{i}x_i$. Let $A$ be a set in $[C'N]$ coming from \cref{thm:5-point} containing at most $\delta^{c \log(1/\delta)} N$ translates of every dilate of $\varphi(T)$. Then  $\varphi^{-1}(A) \cap [N]^k$ is the desired construction that has at most  $\delta^{c' \log(1/\delta)} N^k$ translates of every dilate of $T$.

We now consider patterns $T$ with affine dimension at least $3$. First let us consider the case where $\bm{0}, \bm e_1, \bm e_2, \bm e_3 \in T$, where $\bm e_i$ is the $i$-th coordinate vector. Let $A \subset [N]^3$ as in \cref{thm:3d-corner}. Then $A \times [N]^{k-3} \subset [N]^k$ has the desired property.

Now for an arbitrary $T$ of affine dimension at least 3, we find four points $v_1, v_2, v_3, v_4 \in T$ with affine span of dimension 3 (i.e., not all lying on a plane). Consider the $d$-dimension lattice $L$ generated by $v_2-v_1, v_3-v_1, v_4-v_1$ as well as $d-3$ other independent integer vectors in the complementary subspace. Let $\Phi$ be the linear transformation sending $\ZZ^d$ to $L$ and the first three coordinate vectors to $v_2-v_1, v_3-v_1, v_4-v_1$. Then the image of $A \times [N]^{k-3} \subset [N]^k$ (from the previous paragraph) under $\Phi$ has the desired property (one loses only a constant factor in size depending on the map $\Phi$).
\qed 

\section{Open problems}
We end on a list of open problems.

\subsection{Quantitative dependence}
Let $N_0(\epsilon)$ be the minimum $N_0$ so that Green's theorem \cite{Green05} on popular difference for 3-APs holds for all $N \ge N_0(\epsilon)$, i.e., for every $\epsilon > 0$, $N \ge N_0(\epsilon)$, and $A \subseteq [N]$ with $|A| \ge \delta N$, there is some  $d > 0$ such that $A$ contains at least $(\delta^3 -\epsilon) N$ different 3-APs  with common difference $d$. Green proved that $N_0(\epsilon)$ is at most an exponential tower of 2's with height $\epsilon^{-O(1)}$. The bounds are due to the use of a Szemer\'edi-type regularity lemma argument adapted to the arithmetic setting. 
Recently, Fox, Pham, and Zhao \cite{FP1,FP2,FPZ} showed that the regularity-type bounds are necessary, namely that $N_0$ grows as an exponential tower of twos of height $\Theta(\log (1/\epsilon))$. 

It remains to be explored the quantitative dependencies of $N_0$ on $\epsilon$ for other popular difference patterns such as 4-APs and 2-dimensional corners.

\subsection{Four-point patterns in one and two dimensions}
\cref{cor:3d-pattern} addresses all patterns with affine dimension $3$ or at least $5$ points. This leaves open the cases of popular differences all 4-point patterns in dimensions 1 and 2 except those resolved positively by Green and Tao~\cite{GrTa10} (namely 4-APs and other patterns of the form $\{0, a, b, a+b\}$).

A particular enticing open case is the axis-aligned square: $\{(0,0),(0,1),(1,0),(1,1)\}$. The 1-dimensional projection of this pattern yields a pattern of the form $\{0, a, b, a+b\}$, and this 1-dimensional pattern has the positivity property used in the 4-AP popular difference proof~\cite{GrTa10}.

\subsection{Finite fields versus integers (or other groups)}
The claim in Mandache's \cref{thm:man} that $m_{\text{convex}} \le M$ for 2-dimensional corners (and also \cref{cor:corner}) is currently only proved for $G = \FF_p^n$ with fixed $p$. It seems likely that the methods can be extended to all abelian groups via Bohr sets, similar to \cite{Green05}. However, this has yet to be worked out.

On the other hand, our counterexample construction \cref{thm:3d-corner} for 3-dimensional corners only works over the integers, and not, say $\FF_p^n$. It remains to explore the problem over other groups.

\end{document}